\documentclass[12pt,english]{amsart}

\usepackage[inner=1in,outer=1in,top=1in,bottom=1in]{geometry}
\usepackage{amssymb}
\usepackage{paralist}
\usepackage{babel}
\usepackage[pdftex]{graphicx}    
\usepackage[leqno]{amsmath}
\usepackage{amsthm}
\usepackage{emptypage}
\usepackage{hyperref}

\usepackage{tikz-cd}

\usepackage[utf8]{inputenc}
\usepackage{enumerate}
\usepackage{amsfonts,mathrsfs}
\usepackage{ulem}
\usepackage{paralist}
\usepackage[leqno]{amsmath}
\usepackage{graphics}
\usepackage[absolute]{textpos}

\newcommand\mbb{\mathbb}

\newcommand\C{\mbb{C}}

\newcommand\R{\mbb{R}}

\DeclareMathOperator*{\Mat}{Mat}

\DeclareMathOperator*{\adj}{adj}

\DeclareMathOperator*{\Sym}{Sym}

\renewcommand\epsilon{\varepsilon}

\renewcommand\mod{\mathop{\rm mod}}
\renewcommand\phi{\varphi}

\renewcommand\theta{\vartheta}
\theoremstyle{plain}
\newtheorem{Thm}{Theorem}
\newtheorem{Prop}[Thm]{Proposition}
\newtheorem{Cor}[Thm]{Corollary}
\newtheorem{Lemma}[Thm]{Lemma}

\newtheorem{Conjecture}[Thm]{Conjecture}
\newtheorem*{Thm*}{Theorem}
\newtheorem*{Prop*}{Proposition}
\newtheorem*{Cor*}{Corollary}
\newtheorem*{Lemma*}{Lemma}
\newtheorem*{Sublemma*}{Sublemma}
\newtheorem*{Conjecture*}{Conjecture}
\theoremstyle{definition}

\newtheorem{Example}[Thm]{Example}

\newtheorem{Question}[Thm]{Question}

\newtheorem*{Constr*}{Construction}
\newtheorem*{Def*}{Definition}
\newtheorem*{Defs*}{Definitions}
\newtheorem*{Example*}{Example}
\newtheorem*{Examples*}{Examples}
\newtheorem*{Exercise*}{Exercise}
\newtheorem*{LemmaDef*}{Lemma and Definition}
\newtheorem*{Notation*}{Notation}
\newtheorem*{Problem*}{Problem}
\newtheorem*{Question*}{Question}
\newtheorem*{Remark*}{Remark}
\newtheorem*{Remarks*}{Remarks}
\newtheorem*{Warning*}{Warning}

\newtheorem*{Text*}{}
\numberwithin{equation}{section}
\numberwithin{Thm}{section}

\begin{document}

\title{Two results on the size of spectrahedral descriptions}

 \author{Mario Kummer}
 \address{Universit\"at Konstanz, Germany} 
 \email{Mario.Kummer@uni-konstanz.de}

\begin{abstract}
 A spectrahedron is a set defined by a linear matrix inequality. Given a spectrahedron we are interested in the question
 of the smallest possible size $r$ of the matrices in the description by linear matrix inequalities.
 We show that for the $n$-dimensional unit ball $r$ is at least $\frac{n}{2}$.
 If $n=2^k+1$, then we actually have $r=n$. The same holds true for any compact convex set in $\R^n$ defined by a quadratic polynomial.
 Furthermore, we show that for a convex region in $\R^3$ whose algebraic boundary is smooth and defined by a cubic polynomial
 we have that $r$ is at least five. More precisely, we show that if $A_1,A_2,A_3 \in \Sym_r(\R)$ are real symmetric matrices such that
 $f(x,y,z)=\det(I+A_1 x+A_2 y+A_3 z)$ is a cubic polynomial, 
 the surface in complex
 projective three-space with affine equation $f(x,y,z)=0$ is singular.
\end{abstract}

\maketitle

\section{Introduction}
A \textit{spectrahedron} is the solution set of a linear matrix inequality, i.e. a  set of the form 
\[C=\{p\in \R^{n}: \,\,A(p)= A_0 +A_1  p_1 \ldots + A_n p_n \succeq 0, \,\, \textrm{i.e. } A(p) \textrm{ is positive semidefinite}\} \]
where $A_0,\ldots , A_n \in \Sym_r(\R)$ are real symmetric matrices. Spectrahedra are closed convex and semi-algebraic sets.
They are exactly the feasible sets of semidefinite programming (SDP).
For practical application the question of the size of the matrices in the description of the spectrahedron is very important.
Thus one is interested in the following  question:
\begin{Question}\label{qut:haupt}
 Given a spectrahedron $C \subseteq \R^n$, what is the smallest size  $r$ of real symmetric matrices  $A_0,\ldots , A_n \in \Sym_r(\R)$
 such that we have
  \[C=\{p\in \R^{n}: \,\,A(p)= A_0 +A_1  p_1 \ldots + A_n p_n \succeq 0\} ?\]
\end{Question}
A natural lower bound for the size of matrices is given by the degree of the \textit{algebraic boundary} of $C$, i.e. the Zariski closure of the
Euclidean boundary of $C$. Conversely, it follows from the Theorem of Helton and Vinnikov \cite[Thm. 2.2]{HelVin07}
that every spectrahedron in $\R^2$ with nonempty interior can be described using $d \times d$ matrices where $d$ is the degree of the 
algebraic boundary of $C$.
There are several immediate ways to see that the analogous statement in general fails for $C \subseteq \R^n$ when $n\geq 3$.

Let $P$ be a full-dimensional polytope given by $d$ linear inequalities. Then $P$ has a description as a spectrahedron using $d \times d$ matrices.
 There has been recent interest in finding the smallest number $r$ such that $P$ is the \textit{projection} of a spectrahedron
 given by $r \times r$ matrices, see e.g. \cite{fawzi2014positive, MR3062007, gouv13, lee2014some}. This number is the \textit{positive semidefinite rank} of the \textit{slack matrix}
 of $P$. Good lower bounds are known for example for the cut, TSP, and stable set polytopes \cite{lee2014lower}.

 In the general case there are almost no results regarding Question \ref{qut:haupt}. Already for the unit ball in $\R^n$ the answer is not known.

In this paper, we address this question for compact convex sets in $\R^n$ defined by a quadratic polynomial and for convex sets in $\R^3$
defined by a cubic polynomial.
In section \ref{sec:quad} we consider the unit ball in $\R^n$
which is known to have a description as a spectrahedron using
$n \times n$ matrices. We will show that one can not do better than matrices of size $\frac{n}{2}$.
More precisely, we will show for any natural number $k$
that whenever $n>2^k$ we will need matrices of size larger than $2^k$.
In particular, for $n=2^k+1$ the description using $n \times n$ matrices is the smallest possible description.
The proof relies on well-known results about quadratic forms. 

In section \ref{sec:cub} we consider  convex regions in $\R^3$ 
whose algebraic boundary is defined by a cubic polynomial.
Such a set is a spectrahedron and we show that in general it is impossible to find a description with matrices of size smaller than five.
More precisely, we show that if $A_1,A_2,A_3 \in \Sym_r(\R)$ are real symmetric matrices of size $r$ such that
 \[f(x,y,z)=\det(I_r+A_1 x+A_2 y+A_3 z)\] is a cubic polynomial, then 
 the surface in complex
 projective three-space with affine equation $f(x,y,z)=0$ is singular. The proof crucially uses the characterization of vector spaces of matrices of rank at most
 three \cite{Atki,EH}.

\section{Quadratic polynomials} \label{sec:quad}
The goal of this section is to give lower bounds for the size of matrices in spectrahedral representations of convex sets
defined by a quadratic polynomial. The following theorem is the main result of this section.
\begin{Thm} \label{thm:quad}
 Let $f \in \R[x_1, \ldots, x_n]$ be a quadratic polynomial such that 
 \[C=\{p \in \R^n: \,\, f(p) \geq 0\}\] is a compact, convex set with nonempty interior.
 Consider a spectrahedral representation 
 \[C=\{p\in \R^{n}: \,\,A(p)= A_0 +A_1  p_1 \ldots + A_n p_n \succeq 0\} \]
where $A_0,\ldots , A_n \in \Sym_r(\R)$ are real symmetric matrices of size $r$.
If $n > 2^k$ for some non-negative integer $k$, then $r > 2^k$.

In particular, we have $r \geq \frac{n}{2}$ and if $n=2^k+1$ we even have $r \geq n$.
\end{Thm}

Let $f \in \R[x_1, \ldots, x_n]$ be a quadratic polynomial such that 
 $C=\{p \in \R^n: \,\, f(p) \geq 0\}$ is a compact, convex set with nonempty interior.
After some affine change of coordinates we have that $f=1-(x_1^2+\ldots+x_n^2)$ and $C$ is the unit ball.
Thus we can and will restrict our attention to the case of the unit ball.
There is the following well-known description of the unit ball in $\R^n$ as a spectrahedron using $n \times n$ matrices:
 \[ \begin{pmatrix}
     1+p_1 & p_2 & p_3 &  \ldots & p_n \\
     p_2 & 1-p_1 & 0 & \ldots & 0 \\
     p_3 & 0 & 1-p_1 & \ddots & \vdots \\
     \vdots &\vdots &\ddots & \ddots & 0 \\ 
     p_n &0&\ldots&0& 1-p_1
    \end{pmatrix} \succeq 0.
 \]
Thus when $n=2^k+1$ the lower bound from Theorem \ref{thm:quad} is sharp.

For the proof we need some lemmas concerning minors of a matrix polynomial.

\begin{Lemma}\label{lem:minor1}
 Let $A_1, \ldots, A_n \in \Sym_r(\R)$ and consider $M=I_r x_0 -(A_1 x_1+ \ldots +A_n x_n)$.
 Let $h \in \R[x_0,\ldots,x_n]$ be an irreducible polynomial whose set of real zeros  lies Zariski dense in its complex zero set.
 Let $m \geq 1$.
 If the polynomial $\det(M)$ is divisible by $h^m$,
 then every $(r-m+1) \times (r-m+1)$ minor
 of $M$ is divisible by $h$.
\end{Lemma}

\begin{proof}
 By Hilbert's Nullstellensatz and since $h$ is irreducible with Zariski dense real zeros it suffices to show that 
 every $(r-m+1) \times (r-m+1)$ minor of $M$ vanishes on the real zero set of $h$.  
  Thus let $a=(a_0,\ldots,a_n) \in \R^{n+1}$ such that $h(a)=0$.
 Consider the univariate polynomial $p=h(t,a_1 \ldots, a_{n}) \in \R[t]$.
 Since $p^m$ divides the characteristic polynomial of the matrix $A(a)=a_1 A_1+\ldots+a_{n} A_n$
 and since $p$ has a root at $t=a_0$ the kernel of $M(a)=a_0 I_r-A(a)$ is at least $m$-dimensional.
 Thus every $(r-m+1) \times (r-m+1)$ minor of $M(a)$ vanishes. 
\end{proof}

\begin{Lemma}\label{lem:minor2}
 Let $A_1, \ldots, A_n \in \Mat_{r}(\R)$ be square matrices of size $r$
 and consider \[M=I_r x_0 -(A_1 x_1+ \ldots +A_n x_n).\]
 Let $h \in \R[x_0,\ldots,x_n]$ be an irreducible polynomial whose set of real zeros  lies Zariski dense in its complex zero set.
 Let $m \geq 1$.
 If the polynomial $\det(M)$ is not divisible by $h^m$,
 then there is an $(r-m+1) \times (r-m+1)$ minor which is not divisible by $h$.
 
 If furthermore the matrices $A_1, \ldots, A_n$ are symmetric, then there is a symmetric $s \times s$ minor of $M$ which is not divisible by $h$ with $s \geq r-m+1$.
\end{Lemma}

\begin{proof}
 Let $\det(M)=h^d q$ where $q$ is a polynomial which is coprime to $q$ and $d<m$.
 We observe that there is an $a=(a_0,\ldots,a_n) \in \R^{n+1}$ such that $h(a)=0$ but $q(a) \neq 0$, because otherwise
 $q$ would be divisible by $h$ according to Hilbert's Nullstellensatz. Letting \[g_1=h(t,a_1, \ldots, a_n) \textrm{ and } g_2=q(t,a_1, \ldots, a_n),\]
 the characteristic polynomial of the matrix
 $A(a)=a_1 A_1+\ldots+a_{n} A_n$
 is $g_1^d g_2$.
 Since $g_1$ has a zero at $t=a_0$ but $g_2$ has not the kernel of $M(a)=a_0 I_r-A(a)$ is at most $d$-dimensional.
 This means that $M(a)$ has rank bigger than $r-m$.
 Thus
 there is an $(r-m+1) \times (r-m+1)$ minor of $M(a)$ which does not vanish. If the $A_i$ are symmetric, then there
 is a symmetric $s \times s$ minor of $M(a)$ which does not vanish with $s \geq r-m+1$. This shows the claim.
\end{proof}

\begin{Prop}\label{prop:adjteil}
 Let $A_1, \ldots, A_n \in \Sym_r(\R)$ and consider $M=I_r x_0 -(A_1 x_1+ \ldots +A_n x_n)$.
 Let $h \in \R[x_0,\ldots,x_n]$ be an irreducible polynomial whose set of real zeros  lies Zariski dense in its complex zero set.
 Assume that $\det(M)$ is divisible by $h^m$ but not by $h^{m+1}$ for some $m \geq 1$.
 Then the greatest common divisor
 of all entries of $\adj(M)$ and $h^m$ is precisely
 $h^{m-1}$.
\end{Prop}

\begin{proof}
 By Lemma \ref{lem:minor1} every $(r-m+1) \times (r-m+1)$ minor
 of $M$ is divisible by $h$.
 Thus, if we have an $(r-1) \times (r-1)$ submatrix $M'$ of $M$, then every 
 minor of $M'$ of size $((r-1)-(m-1)+1)$ is divisible by $h$. 
 Now it follows from Lemma \ref{lem:minor2} that $\det(M')$ is divisible by $h^{m-1}$.
 Thus every entry of $\adj(M)$ is divisible by $h^{m-1}$.
 
 On the other hand, Lemma \ref{lem:minor2} sais that there is a symmetric
 $s \times s$ submatrix $M'$ of $M$ whose determinant is not divisible by $h$ with $r-m \leq s < r$.
 This implies in particular that there is a $(r-m) \times (r-m)$ minor of $M'$ that is not divisible by $h$. 
 This is then also an
 $(r-m) \times (r-m)$ minor of a suitable symmetric $(r-1) \times (r-1)$ submatrix  $M''$ of $M$.
  Then
 Lemma \ref{lem:minor1} implies that $\det(M'')$ is not divisible by $h^m$.
\end{proof}

Now we can relate the existence of a spectrahedral representation for the unit ball of a given size to a statement about sum of squares
representations of certain polynomials of a given length.

\begin{Prop}\label{prop:sos}
 Let $C=\{p \in \R^n: \,\, \| p \|_2 \leq 1\}$ be the unit ball. Assume that we have
  \[C=\{p\in \R^{n}: \,\,A(p)= A_0 -A_1  p_1 \ldots - A_n p_n \succeq 0\} \]
where $A_0,\ldots , A_n \in \Sym_r(\R)$ are real symmetric matrices of size $r$.
Then we can find two nonzero polynomials $g_1,g_2 \in \R[x_1, \ldots, x_n]$ that are both a sum of $r$ squares of polynomials
with $(x_1^2+\ldots+x_n^2) \cdot g_1 = g_2$.
\end{Prop}

\begin{proof}
 Without loss of generality we can assume that $A_0=I_r$ is the identity matrix. Then consider $M=I_r x_0 -(A_1 x_1+ \ldots +A_n x_n)$.
 After homogenizing we have that $\det(M)$ is divisible by the polynomial $h=x_0^2-(x_1^2+\ldots+x_n^2)$
 since the $A_i$ give a spectrahedral description of the unit ball.
 Let $\det(M)=q \cdot h^m$ where $m \geq 1$ and $q$ is a polynomial that is coprime to $h$.
 Every entry of $\adj(M)$ is divisible by $h^{m-1}$ and there is at least one entry that is not divisible by $h^m$.
 Let $v$ be a column of $\adj(M)$ that has such an entry. The vector $w=h^{1-m} \cdot v$ has polynomials as entries and
 not every entry of $w$ is divisible by $h$. We have that
 \[
  M \cdot w = h^{1-m} \cdot M \cdot v = h^{1-m} \cdot \det(M) \cdot e = q \cdot h \cdot e
 \]
 where $e \in \R^r$ is a unit vector. Modulo $h$ we can write $w=a \cdot x_0 + b$ where the entries of $a$ and $b$ are
 in $\R[x_1,\ldots,x_n]$. Since not every entry of $w$ is divisible by $h$ we have that $a$ and $b$ are not both the zero vector.
 Letting $A= A_1 x_1+ \ldots +A_n x_n$ we get from $M \cdot w=q \cdot h \cdot e$:
 \[b \cdot x_0 + (x_1^2+\ldots+x_n^2) \cdot a \equiv A \cdot (a \cdot x_0 + b) \,\,\, \mod \, h.\]
 Since the entries of $a,b$ and $A$ are polynomials in $x_1,\ldots,x_n$ and not in $x_0$ we obtain
 by comparing coefficients $A \cdot a= b$ and $A \cdot b = (x_1^2+\ldots+x_n^2) \cdot a$.
 This implies $b^{\rm T} \cdot A \cdot a= b^{\rm T} \cdot b$ and $a^{\rm T} \cdot A \cdot b = (x_1^2+\ldots+x_n^2) \cdot a^{\rm T} \cdot a$.
 Since $A$ is symmetric we get \[b^{\rm T} \cdot b= (x_1^2+\ldots+x_n^2) \cdot a^{\rm T} \cdot a\] which implies the claim.
\end{proof}

\begin{proof}[Proof of Theorem \ref{thm:quad}]
 Let $C \subseteq \R^n$ be the unit ball and
 assume that we have
  \[C=\{p\in \R^{n}: \,\,A(p)= A_0 +A_1  p_1 \ldots + A_n p_n \succeq 0\} \]
where $A_0,\ldots , A_n \in \Sym_r(\R)$ are real symmetric matrices of size $r$.
Let $n > 2^k$ for some non-negative integer $k$. 
Assume for the sake of contradiction that $r \leq 2^k$.
By Proposition \ref{prop:sos} we can find two nonzero polynomials $g_1,g_2 \in \R[x_1, \ldots, x_n]$ that are both a sum of $2^k$ squares of polynomials
with $(x_1^2+\ldots+x_n^2) \cdot g_1 = g_2$. Since the nonzero sums of $2^k$ squares in the rational function field 
$\R(x_1, \ldots, x_n)$ form a multiplicative group \cite[Chapter X, Cor. 1.9.]{LamBook}
we have that $x_1^2+\ldots+x_n^2= g_1^{-1} \cdot g_2$ is a sum of $2^k$ squares of rational functions.
But this is impossible since $n> 2^k$, see \cite[Chapter IX, Cor. 2.4]{LamBook}.
\end{proof}

\begin{Example}
 The key argument was that $(x_1^2+\ldots+x_n^2) \cdot g_1 = g_2$ is impossible when $g_1$ and $g_2$ are both sum of $r=2^k$ squares but $n>r$.
 For arbitrary $r$ this is not true anymore:
 \[
  (x_1^2+x_2^2+x_3^2+x_4^2) \cdot (x_1^2+x_2^2)=(x_1^2+x_2^2)^2+(x_1 x_3 + x_2 x_4)^2+(x_1 x_4 - x_2 x_3)^2.
 \]
 Here $g_1$ and $g_2$ are both sums of $r=3$ squares and $n=4$. 
 This may suggest that one can describe the unit ball $B$ in $\R^4$ using $3 \times 3$ matrices.
 But that is not the case.
 Assume for the sake of contradiction that there is such a representation
 \[B=\{p\in \R^{4}: \,\,A(p)= A_0 +A_1  p_1 + A_2 p_2+ A_3 p_3+ A_4 p_4 \succeq 0\} \]
 with real symmetric matrices $A_0,A_1,A_2,A_3,A_4 \in \Sym_3(\R)$. This implies that
 \[l \cdot (x_0^2-(x_1^2+x_2^2+x_3^2+x_4^2)) = \det(x_0 A_0 + \dots + x_4 A_4)\]
 for some linear form $l \in \R[x_0,\ldots,x_4]$.
 After a linear change of variables, we have \[x_0 \cdot h = \det(x_0 M_0 + \dots + x_4 M_4)\]
 where $M_0,\ldots, M_4 \in \Sym_3(\mathbb{\R})$ and $h$ is a homogeneous polynomial of degree two.
 Furthermore, the matrices $M_0, \ldots, M_4$ are linearly independent since the projective zero set of 
 $x_0^2-(x_1^2+x_2^2+x_3^2+x_4^2)$ is smooth. Letting $x_0=0$ we see that every matrix in the linear span $V$ of 
 $M_1, \ldots, M_4$ has rank at most two. But since $\dim V =4$ that implies that
 $M_1, \ldots, M_4$ either have a common kernel or a common image, see \cite[Exercise 9.24]{Har95} or \cite[Thm. 1.1]{EH}.
 But because these matrices are symmetric that contradicts their linear independence.
 Thus a statement about vector spaces of matrices of rank at most  two gives us a sharp lower bound for $n=4$.
 We will use the same kind of argument in the next section
 to give lower bounds in the case of spectrahedra whose algebraic boundary is defined by a cubic polynomial.
\end{Example}

\section{Cubic polynomials} \label{sec:cub}
Now consider a cubic polynomial $f \in \R[x,y,z]$ which defines a convex region in $\R^3$. It follows from a result in \cite{Buck07} that this
convex region has a  spectrahedral description with $6 \times 6$ matrices. We will show that in most cases, there
is no such representation with $4 \times 4$ (or smaller) matrices.
In fact we will prove the following stronger statement:

\begin{Thm}{\label{thm:haupt}}
 Let $A,B,C \in \Sym_r(\R)$ be real symmetric matrices of size $r$ such that
 \[f=\det(I+Ax+By+Cz)\] is a cubic polynomial. Then the surface in complex
 projective three-space with affine equation $f=0$ is singular.
\end{Thm} 

\begin{Cor}\label{cor:no4x4}
 Let $C \subseteq \R^3$ be a convex region with nonempty interior whose algebraic boundary is defined by a cubic polynomial
 $f \in \R[x,y,z]$ such that the surface in complex
 projective three-space with affine equation $f=0$ is smooth. Then there is no description
 \[C=\{(x,y,z) \in \R^3:\, A_0+A_1x+A_2y+A_3z \succeq 0 \}\]
 for some real symmetric matrices $A_0,A_1,A_2,A_3 \in \Sym_{4}(\R)$.
\end{Cor}

\begin{proof}
 Assume that there is such a description.
 Let $f^h(w,x,y,z)$ be the homogenization of $f$ using the new variable $w$. Then the assumption implies that we have
 \[l(w,x,y,z) \cdot f^h(w,x,y,z) = \det(A_0w+A_1x+A_2y+A_3z)\] for some linear form $l \in \R[w,x,y,z]_1$.
 After linear changes of coordinates we can assume that $l=w$ and $A_0=I$.
 Then we have \[f=f^h(1,x,y,z)= \det(I+A_1x+A_2y+A_3z)\] which is a contradiction to the preceding
 theorem.
\end{proof}

\begin{Example}{\label{exp:dericone}}
 Let $P \subseteq \R^4$ be the polyhedral cone generated by $(1,0,1,1)$, $(1,-1,-1,1)$, $(1,1,-1,1)$, $(1,0,1,-1)$, $(1,-1,-1,-1)$ and $(1,1,-1,-1)$.
 The algebraic boundary of $P$ is the zero set of a  polynomial of degree five (a product of linear forms). 
 The second derivative in first coordinate direction of this polynomial defines a convex cone. 
  Let $C \subseteq \R^3$ be the affine slice of that cone obtained by setting the first coordinate equal to one.
  The algebraic boundary
 of $C$ is given by the zero set of the cubic polynomial in three variables
 \[f=10 - 3 x^2 - 6  y - x^2 y - 3  y^2 + y^3 - 
 3  z^2 + y z^2.\]
 One checks that the surface in complex
 projective three-space with affine equation $f=0$ is smooth. Therefore, $C$ has no description as a spectrahedron by $4 \times 4$ matrices.
 
 As mentioned above it has  a spectrahedral description using $6 \times 6$ matrices and it is not known whether that can be achieved with $5 \times 5$ matrices.
\end{Example}

\begin{figure}[h]
 \includegraphics[width=100pt]{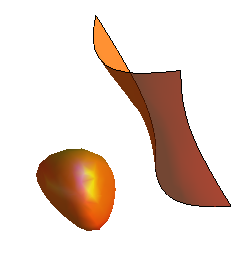}
 \caption{The algebraic boundary of the convex set $C$ in Example \ref{exp:dericone} is smooth. Thus $C$ has a description as a spectrahedron by $6 \times 6$ matrices but not by 
 $4 \times 4$ matrices.}
 \label{fig:dericone}
\end{figure}

\begin{Example}
 Let $e_{3,n}=\sum_{1 \leq i_1<i_2<i_3 \leq n} x_{i_1} x_{i_2} x_{i_3}$ be the elementary symmetric polynomial of degree $3$ in $n \geq 5$
 variables. The dimension of the singular locus of the projective variety \[\{p \in \mathbb{P}^{n-1}:\,\, e_{3,n}(p)=0\}\]
 is less than $n-4$. Therefore, if we intersect with a general three-dimensional subspace we get a smooth variety in $\mathbb{P}^3$ by  Bertini's 
 theorem. As in the proof of Corollary \ref{cor:no4x4}
 there is thus no linear form $l \in \R[x_1, \ldots,x_n]$ such that \[l^N \cdot e_{3,n}= \det(A_1 x_1+\ldots+A_n x_n)\]
 with real symmetric matrices $A_1, \ldots, A_n$ and $A_1v_1+\ldots+A_n v_n \succ 0$ for some point $v \in \R^n$.
 By the same argument we see that there are no real symmetric matrices $A_1, \ldots, A_n$ such that
 $e_{3,n}(1+x_1,\ldots,1+x_n)=\det(I+A_1x_1+\ldots+A_nx_n)$.
\end{Example}
For the proof of Theorem \ref{thm:haupt} we will need the following simple Lemma.
\begin{Lemma}{\label{lem:real}}
 Let $A_1,\ldots, A_n \in \Sym_r(\R)$.
 If  $\det(I+A_1x_1+\ldots+A_nx_n)$ is a polynomial of degree $d$, 
 then $\textnormal{rk}( A_1p_1+\ldots+A_n p_n) \leq d$ for all $p \in \R^n$.
\end{Lemma}

\begin{proof}
 The characteristic polynomial of the matrix
  $M=A_1p_1+\ldots+A_n p_n$  has a root of multiplicity at least $r-d$ at $0$. Since $M$ is symmetric, it has rank at most $d$.
\end{proof}

In the proof of the following important Lemma we use the characterization of vector spaces of matrices of rank at most $3$
which was found independently by the authors of \cite{Atki} and \cite{EH}.

\begin{Lemma}{\label{lem:eh}}
 Let  $A_1,A_2,A_3 \in \Sym_r(\R)$ be symmetric matrices such that the matrix of maximal rank
 in their span has rank three. Further assume that $\cap_{i=1}^3 \ker(A_i) =0$.
 Then, after a change of basis, we have
 \[A_1x+A_2y+A_3z = \begin{pmatrix}
                            0 & \cdots & 0 & \alpha \cdot l_1 & \beta \cdot l_1 \\
                            \vdots & & \vdots & \vdots & \vdots \\
                            0 & \cdots & 0 & \alpha \cdot l_{r-2} & \beta \cdot l_{r-2} \\
                            \alpha \cdot l_1 & \cdots & \alpha \cdot l_{r-2} & g_1 & g_2 \\
                            \beta \cdot l_1 & \cdots & \beta \cdot l_{r-2} & g_2 & g_3
                           \end{pmatrix},
 \] where $l_1, \ldots, l_{r-2}, g_1, g_2, g_3 \in \R[x,y,z]$ are linear forms and $\alpha, \beta \in \R$. In particular $r \leq 5$.
\end{Lemma}

\begin{proof}
 Let $M$ be the span of $A_1,A_2,A_3$ over $\C$.
 Without loss of generality we can assume that $A_1, A_2$ and $A_3$ all have rank three.
 According to \cite[Cor. 1.3]{EH} we are in one of the 
 following situations (notation as in \cite{EH}):
 \begin{enumerate}[a)]
 \item ``$M$ is degenerate'': Since the $A_i$ are symmetric, this is excluded by $\cap_{i=1}^3 \ker(A_i) =0$.
 \item ``$M$ is primitive'': Can be ruled out, since $M$ is spanned by just $3$ matrices.
 \item ``$M$ has primitive part the space of $3 \times 3$ skew-symmetric matrices'':
 In that case we have  $r=4$
 and we can find a  basis $M_1, M_2, M_3$ of $M$ and invertible matrices $R,S$ such that
 \[ R \cdot (M_1 x+ M_2 y+ M_3 z) \cdot S = \begin{pmatrix}
                                                   0 & x & y & 0 \\
                                                   -x & 0 & y & 0 \\
                                                   -y & -z & 0 & 0\\
                                                   l_1 & l_2 & l_3 & l_4
                                                  \end{pmatrix}
 \] for some linear forms $l_1, l_2, l_3, l_4 \in \C[x,y,z]$. Therefore, the matrix
 \[\begin{pmatrix}
                                                   0 & x & y & 0 \\
                                                   -x & 0 & z & 0 \\
                                                   -y & -z & 0 & 0\\
                                                   l_1 & l_2 & l_3 & l_4
                                                  \end{pmatrix} \cdot S^{-1} \cdot R^{\textrm{T}}\]
is symmetric. In particular, if $U$ is the upper left $3 \times3$ submatrix of $   S^{-1}     R^{\textrm{T}} $, then we can
multiply every skew-symmetric $3 \times 3$ matrix to $U$ from the left and obtain a symmetric matrix. But one can easily check that
this implies that $U=0$ which is a contradiction.
  \item ``$M$ is a compression space'': Since $M$ is not degenerate and since the $A_i$ are symmetric,
  there  are linear subspaces $V,W \subseteq \C^r$ with $\dim(V)=r-2$ and $\dim(W)=1$
  such that $A v \in W$ for all $v \in V$ and $A \in M$. In particular, the kernel of each $A_i$ is a subset of
   $V$. Since $\cap_{i=1}^3 \ker(A_i) =0$, it follows that $V$ is spanned by $\ker(A_1)$, $\ker(A_2)$ and $\ker(A_3)$ and that
   $A_i(V)=W$.
   Therefore $V$ and $W$ are defined over $\R$. After a change of basis (over $\R$) we can assume that $V$ is spanned by the first
   $r-2$ unit vectors.  Then $A_1x+A_2y+A_3z$ has the desired form.
 \end{enumerate} 
 The last statement follows since for $r >5$ the $A_i$ would have a common kernel vector.
\end{proof}

\begin{proof}[Proof of Theorem \ref{thm:haupt}]
 The matrices $A, B, C$ span a vector space of matrices of rank at most three by  Lemma \ref{lem:real}.
 After possibly replacing  $A, B$ and $C$
 by smaller matrices we can assume that $A, B$ and $C$ do not have a nonzero common kernel vector.
 According to Lemma \ref{lem:eh} the matrix pencil $Ax+By+Cz$ has, after choosing a suitable basis, the following form:
  \[Ax+By+Cz = \begin{pmatrix}
                            0 & \cdots & 0 & \alpha \cdot l_1 & \beta \cdot l_1 \\
                            \vdots & & \vdots & \vdots & \vdots \\
                            0 & \cdots & 0 & \alpha \cdot l_{r-2} & \beta \cdot l_{r-2} \\
                            \alpha \cdot l_1 & \cdots & \alpha \cdot l_{r-2} & g_1 & g_2 \\
                            \beta \cdot l_1 & \cdots & \beta \cdot l_{r-2} & g_2 & g_3
                           \end{pmatrix},
 \] where $l_1, \ldots, l_{r-2}, g_1, g_2, g_3 \in \R[x,y,z]_1$ are linear forms and $\alpha, \beta \in \R$.
 
 Now let $p \in \C^3$ such that $g_2(p)=\alpha \beta (l_1(p)^2+\ldots+l_{r-2}(p)^2)$, $\beta g_1(p)=\alpha g_2(p)-\beta$
 and $\alpha g_3(p)=\beta g_2(p)-\alpha$. For generic data there are exactly two distict points that satisfy these equations.
 One directly verifies that $I+A_1 p_1+ A_2 p_2+ A_3 p_3=S  S^{\textnormal{T}}$ where 
  \[S=
                           \begin{pmatrix}
                            1 & \cdots & 0  \\
                            \vdots & & \vdots  \\
                            0 & \cdots & 1  \\
                            \alpha \cdot l_1 & \cdots & \alpha \cdot l_{r-2} \\
                            \beta \cdot l_1 & \cdots & \beta \cdot l_{r-2} 
                           \end{pmatrix}.\]
 Since $\textnormal{rk}(S)=r-2$ the matrix $S  S^{\textnormal{T}}$ has two dimensional kernel. Thus the zero set of $f$
 is singular at the point $p$.
\end{proof}

\begin{Cor}
  Let $A,B,C \in \Sym_r(\R)$ be real symmetric matrices of size $r$ such that
 \[f=\det(I+Ax+By+Cz)\] is a cubic polynomial. Then we can find $A',B',C' \in \Sym_5(\R)$
 such that
 \[f=\det(I+A'x+B'y+C'z).\]
\end{Cor}

\begin{proof}
 In the proof of Theorem \ref{thm:haupt} we have seen that one can replace $A,B$ and $C$ by matrices such that their span
 is of the form as in Lemma \ref{lem:eh}. Thus we can bound their size by five.
\end{proof}

\begin{Example}\label{exp:twosing}
Let $C \subseteq \R^3$ be a convex region with nonempty interior whose algebraic boundary is defined by a cubic polynomial
 $f \in \R[x,y,z]$. Let $C$ have a description as a spectrahedron using $4 \times 4$ matrices.
 We have seen that the surface defined by $f$ in complex projective three-space will have generically two singularities.
 These may be real or not: 
 \begin{itemize}
  \item  Let $C_1 \subseteq \R^3$ be the solution set of the following linear matrix inequality:
  \[\begin{pmatrix}
     1 & 0 & x & -x \\
     0 & 1 & y & -y \\
     x & y & 1+x & x+z \\
     -x & -y & x+z & 1-y+z
    \end{pmatrix} \succeq 0.\]
  The algebraic boundary of $C_1$ has exactly two nodes and both of them are real.
 \item  Let $C_2 \subseteq \R^3$ be the solution set of the following linear matrix inequality:
  \[\begin{pmatrix}
     1 & 0 & \sqrt{2} \cdot x & x \\
     0 & 1 & \sqrt{2} \cdot y & y \\
     \sqrt{2} \cdot x & \sqrt{2} \cdot y & 1+x-y & \sqrt{2} \cdot z \\
     x & y & \sqrt{2} \cdot z & 1+x+z
    \end{pmatrix} \succeq 0.\]
    The algebraic boundary of $C_2$ has exactly two nodes and none of them is real.
 \end{itemize}
 Note that both $C_1$ and $C_2$ do not admit a spectrahedral representation by $3 \times 3$ matrices because Cayley's cubic
 symmetroid \cite{cay1869} has four nodes.
\end{Example}

In order to proof Theorem \ref{thm:haupt} we used the characterization of vector spaces of matrices of rank at most three.
There is no such characterization known for rank more than three, but a study of some example suggests that 
Theorem \ref{thm:haupt} could be true for higher degree too.

\begin{Conjecture}
  Let $A,B,C \in \Sym_r(\R)$ be real symmetric matrices of size $r$ such that
 \[f=\det(I+Ax+By+Cz)\] is a polynomial of degree at least three. Then the surface in complex
 projective three-space with affine equation $f=0$ is singular.
\end{Conjecture}

\begin{figure}[h]
 \includegraphics[width=250pt]{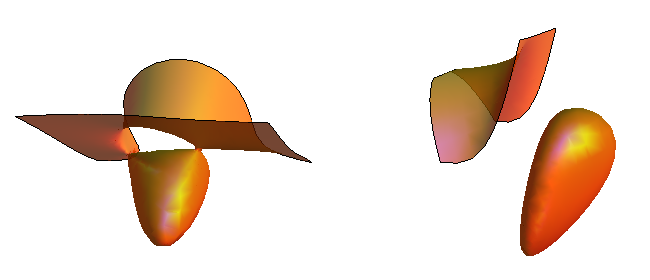}
 \caption{The algebraic boundary of the convex set $C_1$ in Example \ref{exp:twosing} has two real singularities
 whereas the algebraic boundary of $C_2$ has no real singularities but two complex singularities.}
 \label{fig:tworeal}
\end{figure}

\noindent \textbf{Acknowledgements.}
This work is part of my PhD thesis. I would like
to thank my advisor Claus Scheiderer for his encouragement and the
Studienstiftung des deutschen Volkes for their financial and ideal support.
I thank Bernd Sturmfels for helpful remarks and discussions.
I also thank Greg Blekherman, Jean B. Lasserre, Pablo Parrilo and Cynthia Vinzant.

\bibliographystyle{plain}
\bibliography{paper}

\end{document}